\documentclass[11pt]{article}
\usepackage[english]{babel}
\usepackage{times}
\usepackage{amsmath,amstext,amssymb,amsthm, amsfonts}
\usepackage[dvips]{graphics,color}
\newcommand{\bi}{\begin{itemize}}  
\newcommand{\ei}{\end{itemize}}     
\newcommand{\bc}{\begin{center}}  
\newcommand{\ec}{\end{center}}     

\newcommand{\ls}[1]
   {\dimen0=\fontdimen6\the\font \lineskip=#1\dimen0
   \advance\lineskip.5\fontdimen5\the\font \advance\lineskip-\dimen0
   \lineskiplimit=.9\lineskip \baselineskip=\lineskip
   \advance\baselineskip\dimen0 \normallineskip\lineskip
   \normallineskiplimit\lineskiplimit \normalbaselineskip\baselineskip
   \ignorespaces }

\numberwithin{equation}{section}

\newcommand{\slim} {\mathop{\rm lim\,sup}}
\newcommand{\ilim} {\mathop{\rm lim\,inf}}

\newtheorem{lemma}{Lemma}[section]
\newtheorem{theorem}[lemma]{Theorem}
\newtheorem{corollary}[lemma]{Corollary}

\newtheorem{example}[lemma]{Example}

\newtheorem{remark}[lemma]{Remark}

\def\R{\mathbb{R}}
\def\P{\mathbb{P}}

\def\B{\mathcal{B}}

\title{Fatou's Lemma for Weakly Converging Probabilities}

\begin{document}


\date{}
\maketitle

\begin{center}
  Eugene~A.~Feinberg \footnote{Department of Applied Mathematics and
Statistics,
 Stony Brook University,
Stony Brook, NY 11794-3600, USA, eugene.feinberg@sunysb.edu}, \
Pavlo~O.~Kasyanov\footnote{Institute for Applied System Analysis,
National Technical University of Ukraine ``Kyiv Polytechnic
Institute'', Peremogy ave., 37, build, 35, 03056, Kyiv, Ukraine,\
kasyanov@i.ua.}, and Nina~V.~Zadoianchuk\footnote{Institute for
Applied System Analysis, National Technical University of Ukraine
``Kyiv Polytechnic Institute'', Peremogy ave., 37, build, 35,
03056, Kyiv, Ukraine,\
ninellll@i.ua.} \\

\bigskip
Submitted on May 17, 2012; revised on March 25, 2013; April 29, 2013
\end{center}

\begin{abstract}
Fatou's lemma states under appropriate conditions that the
integral of the lower limit of a sequence of functions is not
greater than the lower limit of the integrals.  This note
describes similar inequalities when, instead of a single measure,
the functions are integrated with respect to different measures
that form a weakly convergent sequence.
\end{abstract}

\sloppy \large

\section{The Inequality for Nonnegative Functions}\label{s2}
Consider a measurable space $(S,\cal B)$, where $S$ is a metric
space and ${\mathcal B}$ is its Borel $\sigma$-field. Let $\P(S)$
be the set of probability measures on $(S,{\mathcal B}(S)).$
According to Fatou's lemma, Shiryaev~\cite{Sh}, for any $\mu\in
\P(S)$ and for any sequence of nonnegative measurable functions
$f_1,f_2,\ldots$
\begin{equation}\label{eq2}\int_S \liminf_{n\to+\infty}f_n(s)\mu(ds)\le \liminf_{n\to+\infty}\int_S f_n(s)\mu(ds).
\end{equation}

A sequence of probability measures $\{\mu_n\}_{n\ge 1}$ from
$\P(S)$ converges weakly to $\mu\in\P(S)$ if for any bounded
continuous function $f$ on $S$
\begin{equation}\label{eq1}\int_S f(s)\mu_n(ds)\to \int_S f(s)\mu(ds) \qquad {\rm as
\quad }n\to+\infty.
\end{equation}
A sequence of probability measures $\{\mu_n\}$ from $\P(S)$
converges setwise to $\mu\in\P(S)$ if (\ref{eq1}) holds for any
bounded measurable function $f$. If $\{\mu_n\}_{n\ge 1} \subset
\mathbb{P}(S)$ converges  setwise to $\mu\in \mathbb{P}(S)$,
according to Royden \cite[p. 231]{Ro}, for any sequence of
nonnegative measurable function $f_1,f_2,\ldots$
\begin{equation}\label{eq3}\int_S \liminf_{n\to+\infty}f_n(s)\mu(ds)\le \liminf_{n\to+\infty}\int_S f_n(s)\mu_n(ds).
\end{equation}
However, this is not true, if $\mu_1,\mu_2,\ldots$ converge weakly
to $\mu$.

Indeed, let $S=[0,1]$, $\mu_n(A)={\bf I}\{1/n\in A\},$
$\mu(A)={\bf I}\{0\in A\}$ for $A\in{\cal B}([0,1])$, and
$f(s)=f_n(s)={\bf I}\{s=0 \}$ for $n=1,2,\ldots $ and $s\in
[0,1].$ Then $\int_S f(s)\mu(ds)=1$, $\int_S f(s)\mu_n(ds)=0$, and
(\ref{eq3}) does not hold.

Theorem~\ref{lemma2} presents Fautou's lemma for weakly converging
measures $\mu_n$ and nonnegative functions $f_n$.  This fact is
useful fact for the analysis of Markov decision processes and
stochastic games.  Serfozo~\cite[Lemma 3.2]{Serfozo} establishes
inequality (\ref{eq3.1}) for a vaguelly convergent sequence of
measures on a locally compact metric space $S$ and for nonnegative
functions $f_n.$ In its current form, Theorem~\ref{lemma2} is
formulated in Sch\"al~ \cite[Lemma 2.3(ii)]{Schal} without proof,
in Jaskiewicz and Nowak~\cite[Lemma 3.2]{in} with short
explanations on how the proof from Serfozo~\cite[Lemma
3.2]{Serfozo} can be adapted to weak convergence on metric spaces,
and in Feinberg, Kasyanov, and Zadoianchuk~\cite[Lemma 4]{FKZ}
with a proof.  To make this note logically complete, we provide
the proof of Theorem~\ref{lemma2} in Section~\ref{s3} below. The
provided proof is shorter and simpler than the proof in
\cite{FKZ}. Theorem~\ref{teor3} below extends Theorem~\ref{lemma2}
to functions $f_n$ that can be unbounded from below. Lemma~3.3 in
Jaskiewicz and Nowak~\cite{in} is a particular version of such a
result developed for particular applications in that paper.  Let
$\overline{\mathbb{R}}=[-\infty,+\infty]$.


\begin{theorem}\label{lemma2}
Let $S$ be an arbitrary metric space, $\{\mu_n\}_{n\ge 1} \subset
\mathbb{P}(S)$ converge  weakly to $\mu\in \mathbb{P}(S)$, and
$\{f_n\}_{n\ge 1}$ be a sequence of measurable nonnegative
$\overline{\mathbb{R}}$-valued functions on $S$. Then
\begin{equation}\label{eq3.1}
\int_S \ilim\limits_{n\to+\infty,\, s'\to
s}f_n(s')\mu(ds)\le \ilim\limits_{n\to +\infty}\int_S
f_n(s)\mu_n(ds).
\end{equation}
\end{theorem}

We remark that, if $f_n(s)=f(s)$, $n=1,2,\ldots,$ and the function
$f$ is nonnegative and lower semicontinuous then
$\ilim\limits_{n\to+\infty,\, s'\to
s}f_n(s')=f(s)$ and Theorem~\ref{lemma2} implies that
\begin{equation}\label{eq:J0}
\int_S f(s)\mu(ds)\le \ilim\limits_{n\to +\infty}\int_S
f(s)\mu_n(ds),
\end{equation}
if $\mu_n$ converges weakly to $\mu;$ see Billingsley
\cite[problem 7, Chapter 1, $\S$2]{Bil}, where this fact is stated
for a bounded lower semicontinuous $f$.

Further, for any $\overline{\R}$-valued function $u$ on $S$ we denote
\[
\underline{u}(s)= \ilim_{s'\to s}u(s'), \quad \overline{u}(s)=\slim_{s'\to s}u(s'),\quad s\in S.
\]

Theorem~\ref{teor3} below provides the extended version of Theorem~\ref{lemma2} for unbounded below functions.

\section{Proof of Theorem~\ref{lemma2}}

\begin{proof}
First, we prove the lemma for uniformly bounded above
functions $f_n$.  Let  $f_n(s)\le K<+\infty$ for all $n= 1,2,...$
and all $s\in S$. For $n=1,2,\ldots$ and $s\in S$, define
$F_n(s)=\inf\limits_{m\ge n} f_m(s)$.
The functions $\underline{F}_n:S\to [0,+\infty]$, $n=1,2,\ldots,$
are lower semi-continuous; see,  Berberian \cite[Lemma~5.13.4]{Berb}. 
In addition, for $s\in S$
\begin{equation}\label{equkb0}\underline{F}_n(s)\uparrow \ilim\limits_{n\to+\infty,\, s'\to
s}f_n(s')\qquad {\rm
as\quad
}n\to+\infty.
\end{equation}
By the monotone convergence theorem, 
\begin{equation}\label{eq:J1}
\int\limits_S \ilim\limits_{n\to+\infty,\, s'\to
s}f_n(s')\mu(ds)= \lim\limits_{n\to+\infty}\int\limits_S \underline{F}_n(s)\mu(ds).
\end{equation}
Since the function $\underline{F}_n$, $n=1,2,\dots,$ is lower semi-continuous on $S$ and bounded below and $\mu_m$ converges weakly to $\mu$ as $m\to+\infty$, then formula (\ref{eq:J0}) provides
\begin{equation}\label{eq:J2}
 \int\limits_S \underline{F}_n(s)\mu(ds)\le \liminf_{m\to+\infty}\int\limits_S \underline{F}_n(s)\mu_m(ds),\quad n=1,2,\dots.
\end{equation}
Because of $\underline{F}_n$ is monotonically nondecreasing by $n=1,2,\dots$, then
\begin{equation}\label{eq:J3}
\liminf_{m\to+\infty}\int\limits_S \underline{F}_n(s)\mu_m(ds)\le \liminf_{m\to+\infty}\int\limits_S \underline{F}_m(s)\mu_m(ds),\quad n=1,2,\dots.
\end{equation}
Formulas (\ref{eq:J1})--(\ref{eq:J3}) provide necessary inequality (\ref{eq3.1}).

Thus Theorem~\ref{lemma2} is proved for uniformly bounded functions $f_n$. Consider a sequence $\{f_n\}_{n\ge 1}$ of
measurable nonnegative $\overline{\mathbb{R}}$-valued functions on
$S$. For  $\lambda>0$  set
$f_n^{\lambda}(s):=\min\{f_n(s),\lambda\}$, $s\in S$, $n=
1,2,\ldots\ $. Since the functions $f_n^{\lambda}$ are uniformly
bounded above,
\[
\int_S \ilim\limits_{n\to+\infty,\,
s'\to s}f_n^{\lambda}(s')\mu(ds)\le \ilim\limits_{n\to
+\infty}\int_S f_n^{\lambda}(s)\mu_n(ds)\le \ilim\limits_{n\to
+\infty}\int_S f_n(s)\mu_n(ds).
\]
Then, using Fatou's lemma,
\[
\int_S \ilim\limits_{n\to+\infty,\,
s'\to s}f_n(s') \mu(ds)\le
\ilim\limits_{\lambda\to+\infty} \int_S \ilim\limits_{n\to+\infty,\,
s'\to s}f_n^{\lambda}(s')
\mu(ds).
\]
\end{proof}

\section{A Counterexample for Functions Unbounded
Below}\label{s3} A suitable assumption concerning the negative
parts of the sequence $f_1$, $f_2$, ... of functions is necessary
for Fatou's lemma for weakly converging probabilities as well as
for setwise converging probabilities, as the following example
shows.

\begin{example}\label{exa1} {\rm
The sequence of probability measures $\{\mu_n\}_{n\ge 1}$
converges setwise (and therefore converges weakly) to a
probability measure $\mu$ from $\P(S)$, real function $f:S\to \R$
is continuous,
\[
\int |f(s)|\mu(ds), \ \int |f(s)|\mu_n(ds)<+\infty,\quad n\ge1,
\]
and
\[
\int f(s)\mu(ds)>\lim\limits_{n\to+\infty}\int f(s)\mu_n(ds).
\]

Let $S$ denote the semiinterval $(0,1]$ with the Borel
$\sigma$-field $\B(S)$. For every natural number $n$ define
probability measure
\[
\mu_n(A)=\sqrt{n}\lambda\left(A\cap\left[\frac1{2n},\frac1n\right]\right)+\left(2-\frac{1}{\sqrt{n}}\right)
\lambda\left(A\cap\left[\frac12,1\right]\right),\quad A\in\B(S),
\]
where $\lambda$ is the Lebesgue measure on $(0,1]$. Define also
continuous on $S$ real function $f(s)=-s^{-1}$. The sequence of
probability measures $\{\mu_n\}_{n\ge 1}$ converges setwise (and
therefore converges weakly) to the probability measure $\mu$ from
$\P(S)$, where
$\mu(A)=2\lambda\left(A\cap\left[\frac12,1\right]\right),$
$A\in\B(S),$  and
\[
\int f(s)\mu(ds)=-2\ln(2),\quad \int
f(s)\mu_n(ds)=-\ln(2)\left(\sqrt{n}+2-\frac{1}{\sqrt{n}}\right),\quad
n\ge 1.
\]
Thus}
\[
\int f(s)\mu(ds)>\lim\limits_{n\to+\infty}\int
f(s)\mu_n(ds)=-\infty.
\]
\end{example}
\begin{remark}
If we set  $f(s)=s^{-1}$ for $s\in (0,1]$, $n\ge 1$, in
example~\ref{exa1}, then inequalities  (\ref{eq3}) and
(\ref{eq3.1}) are strict.
\end{remark}

\section{Extensions and Variations} 

In the rest of this paper, we deal with integrals of functions
that can take negative values.  An integral $\int_Sf (s)\mu(ds)$
of a measurable $\overline{\R}$-valued function $f$ on $S$ with
respect to a probability measure $\mu \in \P(S)$ is defined if
\begin{equation}\label{e:condint} \min\{ \int_Sf^+(s)\mu(ds),
\int_Sf^-(s)\mu(ds)\}< +\infty, \end{equation} where
$f^+(s)=\max\{f(s),0\}$, $f^-(s)=-\min\{f(s),0\}$, $s\in S$.  If
\eqref{e:condint} holds then the integral  is defined as \[\int_Sf
(s)\mu(ds)=\int_Sf^+ (s)\mu(ds)- \int_Sf^- (s)\mu(ds).\] All the
integrals in the assumptions of the following theorems and
corollary are assumed to be defined. For example, by writing
\eqref{eq:sw1} in Theorem~\ref{teor2}, we assume that the
integrals are defined for the functions $g_n(s),$ $n\ge 1$, and
$\limsup_{n\to+\infty}g_n(s)$. 

The following statement is a generalization of (\ref{eq3}) to
functions that can take negative values.

\begin{theorem}\label{teor2}
Let $\{\mu_n\}_{n\ge 1} \subset \mathbb{P}(S)$ converge  setwise
to $\mu\in \mathbb{P}(S)$ and let  $\{f_n\}_{n\ge 1}$ be a
sequence of measurable $\overline{\R}$-valued functions defined on
$(S,\B(S))$. Then inequality (\ref{eq3}) holds, if all the
integrals in (\ref{eq3}) are defined and there exists a sequence
of measurable ${\R}$-valued functions $\{g_n\}_{n\ge 1}$ on $S$
such that $f_n(s)\ge g_n(s)$, for all $n\ge 1$ and for all $s\in
S$, and
\begin{equation}\label{eq:sw1}
-\infty<\int_S \limsup_{n\to+\infty}g_n(s)\mu(ds)\le\liminf_{n\to+\infty}\int_S
g_n(s)\mu_n(ds).
\end{equation}
\end{theorem}
\begin{proof}
  If at least one of the  inequalities
\begin{equation}\label{eq:fl}
\liminf_{n\to+\infty}\int_S f_n(s)\mu_n(ds)<+\infty,\quad \quad
-\infty<\int_S \liminf_{n\to+\infty}f_n(s)\mu(ds)
\end{equation}
is violated then inequality (\ref{eq3}) holds.  So, we assume
\eqref{eq:fl}. The  left inequality in \eqref{eq:fl} implies
\begin{equation}\label{eq:flg}
\liminf_{n\to+\infty}\int_S g_n(s)\mu_n(ds)<+\infty.
\end{equation}

Let us apply Fatou's lemma for setwise converging probabilities
(see (\ref{eq3})) to the  sequence $\{f_n - g_n\}_{n\ge 1}$ of
nonnegative $\overline{\R}$-valued measurable functions on $S$.
Then
\begin{equation}\label{eq:fl1}
\int_S \liminf_{n\to+\infty}(f_n(s)-g_n(s))\mu(ds)\le \liminf_{n\to+\infty}\int_S (f_n(s)-g_n(s))\mu_n(ds).
\end{equation}

Inequalities (\ref{eq:sw1}) and (\ref{eq:flg}) imply
\begin{equation}\label{eq1gn} -\infty <\int_S\limsup_{n\to+\infty}g_n(s)\mu(ds)<+\infty. \end{equation}
 In view of \eqref{eq1gn}
  and the right inequality in \eqref{eq:fl},
\begin{equation}\label{eq:fl1EF1}
\liminf_{n\to+\infty}f_n(s)-\limsup_{n\to+\infty}g_n(s)\le
\liminf_{n\to+\infty}(f_n(s)-g_n(s))\quad \mu(ds)\mbox{-a.s.},
\end{equation}
and
\begin{equation}\label{eq:fl1EF}
\int_S \liminf_{n\to+\infty}f_n(s)\mu(ds)-\int_S
\limsup_{n\to+\infty}g_n(s)\mu(ds)\le\int_S
\liminf_{n\to+\infty}(f_n(s)-g_n(s))\mu(ds).
\end{equation}
The following inequalities and \eqref{eq1gn} imply \eqref{eq3}
since
\begin{equation*}\begin{aligned}\int_S
\liminf_{n\to+\infty}f_n(s)&\mu(ds)-  \int_S \limsup_{n\to+\infty}
g_n(s)\mu(ds) 
\le \liminf_{n\to+\infty}\int_S (f_n(s)-g_n(s))\mu_n(ds)\\ &\le
\liminf_{n\to+\infty}\int_S
f_n(s)\mu_n(ds)-\liminf_{n\to+\infty}\int_S g_n(s)\mu_n(ds)
\\
&\le \liminf_{n\to+\infty}\int_S f_n(s)\mu_n(ds)  -\int_S
\limsup_{n\to+\infty}g_n(s)\mu(ds), \end{aligned}\end{equation*}
where the first inequality follows from \eqref{eq:fl1EF} and
\eqref{eq:fl1}, the second one holds since $-\infty
<\liminf_{n\to+\infty}\int_S g_n(s)\mu_n(ds)<+\infty$ in view of
\eqref{eq:sw1}, \eqref{eq:fl}, and $g_n\le f_n$, and the last
inequality holds because of \eqref{eq:sw1} and \eqref{eq1gn}.
\end{proof}

\begin{remark}The second inequality in (\ref{eq:sw1}) coincides with (\ref{eq3}),
when $f_n=g_n=g$, $n=1,2,\dots$.
\end{remark}

The following theorem extends Theorem~\ref{lemma2} to functions that can take negative values. 

\begin{theorem}\label{teor3}
Let $S$ be an arbitrary metric space, $\{\mu_n\}_{n\ge 1} \subset
\mathbb{P}(S)$ converge  weakly to $\mu\in \mathbb{P}(S)$, and
$\{f_n\}_{n\ge 1}$ be a sequence of measurable
$\overline{\mathbb{R}}$-valued functions on $S$. Then inequality
(\ref{eq3.1}) holds, if all the integrals in (\ref{eq3.1}) are
defined and there exists a sequence of measurable
${\mathbb{R}}$-valued functions $\{g_n\}_{n\ge 1}$ on $S$ such
that $f_n(s)\ge g_n(s)$, for all $n\ge 1$ and for all $s\in S$,
and
\begin{equation}\label{eq:sw2}
-\infty<\int_S \slim\limits_{n\to+\infty,\, s'\to s}g_n(s') \mu(ds)\le\liminf_{n\to+\infty}\int_S
 {g}_n(s)\mu_n(ds).
\end{equation}
\end{theorem}
\begin{proof}
If at least one of the  inequalities
\begin{equation}\label{eq:fll}
\liminf_{n\to+\infty}\int_S f_n(s)\mu_n(ds)<+\infty,\quad -\infty<\int_S \liminf_{n\to+\infty,\,s'\to s}f_n(s')\mu(ds).
\end{equation}
is violated then inequality (\ref{eq3.1}) holds. So, we assume (\ref{eq:fll}). The  left inequality in (\ref{eq:fll}) implies
\begin{equation}\label{eq:flg1}
\liminf_{n\to+\infty}\int_S g_n(s)\mu_n(ds)<+\infty.
\end{equation}
Inequalities (\ref{eq:sw2}) and (\ref{eq:flg1}) imply that
\begin{equation}\label{eq1gn1}
-\infty <\int_S \slim\limits_{n\to+\infty,\, s'\to s}g_n(s') \mu(ds)<+\infty.
\end{equation}
In view of (\ref{eq1gn1}) and the right inequality in (\ref{eq:fll}),
\begin{equation}\label{eq:k1}
\ilim\limits_{n\to+\infty,\, s'\to s} f_n(s')-\slim\limits_{n\to+\infty,\, s'\to s}g_n(s') \le \ilim\limits_{n\to+\infty,\, s'\to
s}\left[f_n(s')-g_n(s')\right]\quad \mu(ds)\mbox{-a.s.},
\end{equation}
and
\begin{equation}\label{eq:un1}
\int_S \ilim\limits_{n\to+\infty,\, s'\to s} f_n(s')\mu(ds)-\int_S \slim\limits_{n\to+\infty,\, s'\to s}g_n(s')\mu(ds)\le \int_S
\ilim\limits_{n\to+\infty,\, s'\to s} h_n(s')\mu(ds),
\end{equation}
where $h_n(s)= f_n(s)-g_n(s)$, $ s\in S$, $n=1,2,\dots\,.$

Let us apply Fatou's lemma for weak converging probabilities (see Theorem~\ref{lemma2}) to the  sequence $\{h_n\}_{n\ge 1}$ of nonnegative
$\overline{\R}$-valued measurable functions on $S$. Then
\begin{equation}\label{eq:un1a}
\int_S \ilim\limits_{n\to+\infty,\, s'\to s}h_n(s')\mu(ds)\le \liminf_{n\to+\infty}\int_S h_n(s)\mu_n(ds).
\end{equation}
Since $-\infty <\liminf_{n\to+\infty}\int_S
g_n(s)\mu_n(ds)<+\infty$ in view of (\ref{eq:sw2}),
(\ref{eq:flg1}), and $g_n\le f_n$, we have
\begin{equation}\label{eq:un2}
\liminf_{n\to+\infty}\int_S h_n(s)\mu_n(ds)
\le \liminf_{n\to+\infty}\int_S
f_n(s)\mu_n(ds)-\liminf_{n\to+\infty}\int_S {g}_n(s)\mu_n(ds).
\end{equation}
The following inequalities (\ref{eq:un1})--(\ref{eq:un2}) and
(\ref{eq1gn1}) imply (\ref{eq3.1}) since
\begin{equation*}\begin{aligned}
\int_S \ilim\limits_{n\to+\infty,\, s'\to s} f_n(s')\mu(ds) &
-\int_S \slim\limits_{n\to+\infty,\, s'\to s}g_n(s')\mu(ds)
\\
\le\liminf_{n\to+\infty}\int_S f_n(s)\mu_n(ds) &
-\liminf_{n\to+\infty}\int_S{g}_n(s) \mu_n(ds)
\\
\le \liminf_{n\to+\infty}\int_S f_n(s)\mu_n(ds) & -\int_S \slim\limits_{n\to+\infty,\, s'\to s}g_n(s')\mu(ds),
\end{aligned}\end{equation*}
where the first inequality follows from (\ref{eq:un1}) and
(\ref{eq:un2}), and the second inequality holds because of
(\ref{eq:sw2}) and (\ref{eq1gn1}).
\end{proof}

\begin{remark}\label{rem:bb}
Observe that, if the functions $f_n(s)\ge K>-\infty$ for any $s\in S$ and $n=1,2,\ldots,$ in Theorem~\ref{teor3}, then $g_n(s)=K$ for any $s\in S$ and
$n=1,2,\ldots,$ and assumption (\ref{eq:sw2}) holds. This fact also follows from Theorem~\ref{lemma2}.
\end{remark}

\begin{remark}
Example~\ref{exa1} demonstrates that assumptions (\ref{eq:sw1}) and (\ref{eq:sw2}) are essential for Theorems~\ref{teor2} and
\ref{teor3} respectively. 
\end{remark}

\begin{remark}
Theorem~\ref{lemma2} yields that, for uniformly bounded above functions $\{g_n\}_{n\ge 1}$, assumption (\ref{eq:sw2}) in Theorem~\ref{teor3}
is equivalent to 
\begin{equation}\label{eq:E}
\int_S \slim\limits_{n\to+\infty,\, s'\to s}g_n(s') \mu(ds)=\lim_{n\to+\infty}\int_S
 {g}_n(s)\mu_n(ds)>-\infty.
\end{equation}

Indeed, applying Fatou's lemma for uniformly bounded below functions $\{-g_n\}_{n\ge 1}$ (see Remark~\ref{rem:bb}) we obtain the inequality
\[
\int_S \slim\limits_{n\to+\infty,\, s'\to s}g_n(s') \mu(ds)\ge \limsup_{n\to+\infty}\int_S
 {g}_n(s)\mu_n(ds),
\]
that together with assumption (\ref{eq:sw2}) imply (\ref{eq:E}).
\end{remark}

\begin{corollary}\label{cor}
Let $S$ be an arbitrary metric space, $\{\mu_n\}_{n\ge 1} \subset
\mathbb{P}(S)$ converge  weakly to $\mu\in \mathbb{P}(S)$, and
$\{f_n\}_{n\ge 1}$ be a sequence of measurable
$\overline{\mathbb{R}}$-valued functions on $S$. Then inequality
(\ref{eq3.1}) holds, if there exists a bounded above measurable
${\mathbb{R}}$-valued function $g$ on $S$ such that $f_n(s)\ge
g(s)$ for all $n\ge 1$ and $s\in S$, and
\begin{equation}\label{eq:sw2aaa}
-\infty<\int_S \overline{g}(s)\mu(ds)=\lim_{n\to\infty}\int_S
g(s)\mu_n(ds).
\end{equation}
\end{corollary}

\begin{remark}
If function $g$ from Corollary~\ref{cor} is upper semi-continuous
(in particular, continuous), then Assumption (\ref{eq:sw2aaa}) has
the following form:
\[
-\infty<\int_S {g}(s)\mu(ds)=\lim_{n\to\infty}\int_S
g(s)\mu_n(ds).
\]
\end{remark}

In the following example functions $\{f_n\}_{n\ge 1}$ are unbounded below and the assumptions of Theorem~\ref{teor3} are satisfied.
\begin{example}
{\rm Let $S=\mathbb{Q}$ be the set of rational numbers with the
metric $\rho(s_1,s_2)=|s_1-s_2|$, $s_1,s_2\in S$. We number the
elements of $S=\{x_i\}_{i\ge 1}$ and set $f_n=g_n=-n{\bf I}\{s\in
D_n\}$, where $D_n=\{x_1,x_2,\dots,x_n\}$, $n=1,2,\dots$. Note
that $\slim\limits_{n\to+\infty,\, s'\to s}g_n(s')=0$ for any
$s\in S$.

We consider an increasing sequence of natural numbers $\{k_n\}_{n\ge 1}\subset \mathbb{N}$ such that $\frac{k_n}{k_n+1}\notin D_n$, $n=1,2,\dots$. Let us
set
\[
\mu_n(B)={\bf I}\left\{\dfrac{k_n}{k_n+1}\in B\right\},\quad \mu(B)={\bf I}\left\{1\in B\right\},\quad B\in{\cal B}(S),\ n=1,2,\dots.
\]
The sequence of probability measures $\{\mu_n\}_{n\ge
1}\subset\P(S)$ converges weakly to $\mu\in \mathbb{P}(S)$.
Moreover, assumption (\ref{eq:sw2}) holds. Therefore,
Theorem~\ref{teor3} implies (\ref{eq3.1}).

We remark that $g(s)=-\infty$ for all $s\in S$ for any function
$g$ such that $g(s)\le f_n(s)$ for all $n=1,2,\ldots$ and for all
$s\in S$. Thus, ${\bar g}(s)=-\infty$ for all $s\in S$, assumption
(\ref{eq:sw2aaa}) does not hold, and Corollary~\ref{cor} is not
applicable to this example.}
\end{example}

\vspace{.3cm}
 {\bf Acknowledgements.} The authors thank Professor M. Z. Zgurovsky for initiating their research cooperation.
 The authors thank  Dr. Huizhen Janey Yu for her useful remarks
during the preparation of this work.
Research of the first
author was partially supported by NSF grants  CMMI-0900206 and
CMMI-0928490.

\end{document}